\documentclass[11pt]{article}
\usepackage{setspace}
\usepackage[left=1in,right=1in,top=1in,bottom=1in]{geometry}
     
\usepackage{amsthm}
\usepackage[autocite=inline, sortcites=true, maxbibnames=99, 
    sorting=nyt, 
    style=numeric-comp
    ]{biblatex}
    \DeclareFieldFormat[article,inbook,incollection,
        inproceedings,patent,thesis,unpublished]
        {title}{\emph{#1\isdot}} 
    \renewbibmacro{in:}{} 
    \DeclareFieldFormat{journaltitle}{\upshape{#1}}
    
    \DeclareFieldFormat{url}{}
    \DeclareFieldFormat[article]{number}{({#1})}
    \DeclareFieldFormat{eid}{#1, }
    \renewbibmacro{volume+number+eid}{%
        \printfield{volume}%
        \printfield{number},%
        \setunit{\space}%
        \printfield{eid}%
    }
    \renewbibmacro*{issue+date}{%
        \setunit*{\addspace}%
        \usebibmacro{date}%
        \newunit%
    }

\usepackage[colorlinks]{hyperref}
    \hypersetup{
        linkcolor=black,
        citecolor=red!70!black, 
        urlcolor=blue!70!black}
\usepackage{amsmath}

\usepackage[nameinlink]{cleveref}
    \crefname{ex}{Example}{Examples}
    \crefname{thm}{Theorem}{Theorems} 
    \crefname{lem}{Lemma}{Lemmas}
    \crefname{prop}{Proposition}{Propositions}
    \crefname{cor}{Corollary}{Corollaries} 
    \crefname{conj}{Conjecture}{Conjectures} 
    \crefname{defn}{Definition}{Definitions}
    \crefname{rmk}{Remark}{Remarks} 
    
	\newtheorem{thm}{Theorem}[section]
	
	\newtheorem{prop}[thm]{Proposition}
	\newtheorem{cor}[thm]{Corollary}
	\newtheorem{conj}[thm]{Conjecture}
	\newtheorem*{thm*}{Theorem}
	\newtheorem*{cor*}{Corollary}
	\theoremstyle{definition} 
		\newtheorem{defn}[thm]{Definition}
		\newtheorem{ex}[thm]{Example}
    	\newtheorem{rmk}[thm]{Remark}	
    \newcommand{\df}[1]{{\bf\emph{{#1}}}}

\usepackage{tikz, graphicx}
\usepackage[margin=1.5cm]{caption}
\usepackage{subcaption}
	\usetikzlibrary{positioning}
	\usetikzlibrary{arrows}
	\usetikzlibrary{decorations.pathmorphing}
	\usepackage{pgfplots}
        \tikzset{%
        fwdrxn/.style={very thick, arrows={-Stealth[length=5pt,width=5pt]}},
        revrxn/.style={very thick, arrows={-Stealth[length=5pt,width=5pt,left]}},
        newt/.style={turq, opacity=0.15}
        }
        \tikzset{near start abs-right/.style={xshift=1cm}}
        \tikzset{near start abs-left/.style={xshift=-3.5cm}}
        \tikzset{near start abs-up/.style={yshift=1.5cm}}
        \tikzset{near start abs-down/.style={yshift=-1cm}}
	
\usepackage{color, xcolor}

	\definecolor{orange}{RGB}{250, 140, 0}
		
	\definecolor{turq}{RGB}{0, 160, 160}
		
	\definecolor{violet}{RGB}{164, 98, 234}
		
    \definecolor{viridisyellow}{RGB}{253,231,36}
    \definecolor{viridisyellowpale}{RGB}{239,223,81}
    \definecolor{viridisgreen}{RGB}{121,209,81}
        \definecolor{hlgreen}{RGB}{16,115,16}
    \definecolor{viridisturq}{RGB}{34,167,132}
    \definecolor{viridisblue}{RGB}{64,67,135}
    \definecolor{viridisviolet}{RGB}{68,1,84}
	
	\definecolor{ratecnst}{RGB}{172,172,172}
	
\usepackage{multirow}
\usepackage{array}

\usepackage{parskip}
\usepackage{amsfonts, amssymb, mathtools} 
\newcommand{\eq}[1]{\begin{align*}#1\end{align*}}
	\newcommand{\eqn}[1]{\begin{align}#1\end{align}}  
	
\newcommand{\st}{\colon}                
  
\newcommand\mc[1]{\mathcal{#1}}

\newcommand{\rr}{\ensuremath{\mathbb{R}}}   
 
\newcommand{\zz}{\ensuremath{\mathbb{Z}}}

\renewcommand{\epsilon}{\varepsilon}	
    \def\eps{\epsilon}                   
\renewcommand{\phi}{\varphi}			

\DeclareMathOperator{\diag}{diag}		
\newcommand{\grad}{D}
\DeclareMathOperator{\Span}{span}		
\newcommand{\kk}{\kappa}

\newcommand{\vv}[1]{{\boldsymbol{#1}}}  
\newcommand{\rrp}{\rr_{\geq}}
\newcommand{\rrpp}{\rr_{>}}
\newcommand{\zzp}{\zz_{\geq}}

\newcommand{\xx}{\vv x}
\newcommand{\yy}{\vv y}
\newcommand{\ratecnst}[1]{{\footnotesize{\color{gray}{#1}}}}

\usepackage{enumitem}
\usepackage[version=4]{mhchem}
\usepackage{chemfig}

\usepackage{authblk}\usepackage[symbol]{footmisc}
\title{
   Global stability of perturbed complex-balanced systems
}
\author[1]{
        Polly Y. Yu%
}
\affil[1]{\small NSF--Simons Center for Mathematical and Statistical Analysis of Biology, Harvard University }
\date{} 

\begin{document}
\maketitle
\renewcommand*{\thefootnote}{\arabic{footnote}}

\begin{abstract}
    A class of polynomial dynamical systems called complex-balanced are locally stable and conjectured to be globally stable. In general, complex-balancing is not a robust property, i.e., small changes in parameter values may result in the loss of the complex-balanced property. We show that robustly permanent complex-balanced systems are globally stable even after the rate constants have been perturbed. 
\end{abstract}

\section{Introduction} 
\label{sec:intro}

Robustness is vital to biological systems. From cells to ecosystems, the dynamical behaviour is prescribed by a set of core interactions, which are influenced by intrinsic and environmental noise. 
It is often observed that the system maintains its stable dynamical behaviour despite these external influences. In other words, the system is robust with respect to noise. 
One way to measure robustness is to ask whether certain qualitative dynamics remains after the parameters have been perturbed.

Common models for chemical and biological systems assume mass-action kinetics. The dynamics comes from a system of ODEs built using a list of reactions or interactions, each with a rate constant. In general, mass-action systems can display a diverse set of dynamics, from multistability, oscillations, and even chaos. A particularly stable family of mass-action systems is that of \emph{complex-balancing}~\cite{HornJackson1972}, introduced as a generalization of systems at thermodynamic equilibrium. As such, these systems have exactly one positive steady states (up to conservation laws), which is asymptotically stable~\cite{HornJackson1972} and conjectured to be globally attracting~\cite{Horn1974_GAC}. 

Not only are complex-balanced systems dynamically stable and algebraically rich, they are completely characterized by their underlying networks. A mass-action system is complex-balanced if and only if the underlying network is weakly reversible and the rate constants satisfy some algebraic equations~\cite{horn1972necessary, ToricDynSys2009}. The number of these equations is dictated by the network's topology and geometry.

In the context of biochemical reactions, rate constants are not truly constant; they depend on temperature, pressure, the presence of solvents and ions, etc., and are subjected to thermal fluctuation. Thus the algebraic condition for complex-balancing is not satisfied in general, i.e., complex-balancing is generally not robust. With small perturbations in the rate constants, we can no longer claim all of  complex-balancing's stable dynamical properties. 

In this work, we show that robustly permanent complex-balanced systems are globally stable even with small changes in the rate constants. 
\begin{thm*}
    Let $(G,\vv\kk^*)$ be a complex-balanced system that is robustly permanent with respect to $\vv\kk^*$. Then on every compatibility class $U$, there exists $\eps > 0$ such that  for every $\vv \kk \in B(\vv\kk^*,\eps)$, the mass-action system $(G,\vv\kk)$ has a unique globally attracting point within $U$.
\end{thm*}
It is worth noting that robust permanence for complex-balanced systems follows immediately from the Permanence Conjecture for variable-$\vv\kk$ weakly reversible systems, which has been proved for several special classes; see for example \cite{CraciunNazarovPantea2013GAC, Anderson2011GAC, BorosHofbauer2019, GopalkrishnanMillerShiu2014}. 

This paper is organized as follows. Mathematical notations used throughout this work are covered in \Cref{sec:notation}. We cover the necessary background on mass-action systems, complex-balanced systems, and permanence in \Cref{sec:MAS}. Then we prove our main result and give some examples in \Cref{sec:main}.

\section{Notations}
\label{sec:notation}

Throughout this work, we let $\rrp$ and $\rrpp$ denote the set of non-negative and positive real numbers respectively. Accordingly, $\rrp^n$ and $\rrpp^n$ denote the set of vectors in $\rr^n$ with non-negative and positive entries respectively. We say $\xx > \vv 0$ if $\xx \in \rrpp^n$, and $\xx \geq \vv 0$ is defined analogously. Let $B(\xx,\eps)$ be the $\eps$-ball around $\xx$. 
For any $\xx \in \rrpp^n$ and $\yy \in \rrp^n$, let $\xx^\yy = x_1^{y_1} x_2^{y_2} \cdots x_n^{y_n}$. 

\section{Mass-action systems} 
\label{sec:MAS}

We now provide a brief introduction to reaction networks and mass-action systems. For a more in-depth introduction, see \cite{Gunawardena_review, YuCraciun2018_review, Feinberg_book}.

A \df{reaction network} (or \df{network} for short) is a directed graph $G = (V,E)$ with no self-loops and no isolated vertices, where $V \subset \rrp^n$ and $E \subseteq V \times V$. An edge $(\yy_i,\yy_j) \in E$ is denoted $(i,j)$ or $\yy_i \to \yy_j$. 
A vertex is also called a \df{complex}, while an edge is called a \df{reaction}. Throughout this work, $\rr^n$ will be the ambient space; $m = |V|$, $r = |E|$, and $\ell$ is the number of connected components of $G$. 
A network is said to be \df{weakly reversible} if every connected component is strongly connected. 

In some classical literature on mass-action systems~\cite{Feinberg_book}, a reaction network is defined to be a triple $(\mc S, \mc C, \mc R)$, where $\mc S$ is the set of \emph{species}, $\mc C$ is the set of \emph{complexes}, and $\mc R$ is the set of \emph{reactions}. The two definitions are equivalent via a natural identification between the species and the standard orthonormal basis of $\rr^n$. For example when $n=3$, the complex $\sf{X}$ is identified with $(1,0,0)$ and  $\sf{2Y}+\sf{Z}$ with $(0,2,1)$, and so forth.

Assuming \emph{mass-action kinetics}, the time-evolution of the concentration vector $\xx(t)$ is given by the system of autonomous ODEs
\eqn{ \label{eq:mas}
    \dot{\xx}(t) &=  \sum_{(i,j) \in E} \kk_{ij} \xx(t)^{\yy_i} (\yy_j - \yy_i), 
}
where $\kk_{ij} > 0$ is the \df{rate constant} of the reaction $\yy_i \to \yy_j$. We say \eqref{eq:mas} is the associated system of the \df{mass-action system} $(G,\vv\kk)$, where $\vv \kk = (\kk_{ij})_{(i,j) \in E}$. 

\begin{rmk} 
In this work, we restrict the set of vertices to $V \subset (\{0\} \cup [1,\infty))^n$ or $V \subset \zzp^n$. Under this assumption, the state space $\rrpp^n$ is forward-invariant; moreover, the right-hand side of \eqref{eq:mas} is Lipschitz continuous, so solution to \eqref{eq:mas} with any initial condition in $\rrpp^n$ is unique~\cite{Sontag_TCell}.
\end{rmk} 

The \df{stoichiometric subspace} of a reaction network $G$ is the linear subspace
\eq{ 
    S = \Span\{ \yy_j - \yy_i \st  (i,j) \in E\}, 
}
which contains $\dot{\xx}(t)$. Hence for any initial condition $\xx(0) \in \rrpp^n$, the solution $\xx(t)$ is confined to the \df{compatibility class} $(\xx(0) + S)_> = (\xx(0) + S)\cap \rrpp^n$.  

A generalization of \eqref{eq:mas} is the \df{variable-$\vv\kk$ mass-action system}
\eqn{\label{eq:vk-mas} 
    \dot{\xx}(t) &=  \sum_{(i,j) \in E} \kk_{ij}(t) \xx(t)^{\yy_i} (\yy_j - \yy_i), 
}
where $\eps \leq \kk_{ij}(t) \leq \eps^{-1}$ for some uniformly chosen $0 < \eps < 1$~\cite{CraciunNazarovPantea2013GAC}. One might further assume the coefficient functions $\kk_{ij}(t)$ to be sufficiently smooth, for example Lipschitz. Regardless, the dynamics of \eqref{eq:vk-mas} is also constrained on compatibility classes. Clearly, the dynamics of \eqref{eq:mas} is replicated by \eqref{eq:vk-mas}  when each $\kk_{ij}(t)$ is constant.

\begin{figure}[h!!]
\centering
\begin{subfigure}[b]{0.25\textwidth}
\centering 
    \begin{tikzpicture}[scale=0.9]
    \node (1) at (0,0) [left] {$\sf{3X}$};
    \node (2) at (3,0) [right] {$\sf{3Y}$};
    \node (3) at (1.5,3) {$\sf{3Z}$};
    \node (4) at (1.5,1.)  {$\sf{X+Y+Z}$};
    \draw [fwdrxn] (1)--(4)node [midway, below right=-3pt] {\ratecnst{$\kk_{14}$}};
    \draw [fwdrxn] (4)--(3)node [midway, below right] {\ratecnst{\!$\kk_{43}$}};
    \draw [fwdrxn] (3)--(1)node [midway, above left] {\ratecnst{$\kk_{31}$\!\!}};
    \draw [fwdrxn] (3)--(2)node [midway, above right] {\ratecnst{\!\!$\kk_{32}$}};
    \draw [fwdrxn] (2)--(1)node [midway, below] {\ratecnst{$\kk_{21}$}};
\end{tikzpicture}
\vspace{-12pt}
    \caption{} 
    \label{fig:introEx-network-plain} 
    
    \begin{tikzpicture} 
    \begin{axis}[
      view={115}{15},
      axis lines=center,
      width=2in,height=2in,
      ticks = none, 
      xmin=0,xmax=3.75,ymin=0,ymax=3.5,zmin=0,zmax=3.5,
    ]
    \def\xmax{3.75}
    \def\ymax{3.5}
    \def\zmax{3.5}
    \def\gridcolor{gray!20}
    \foreach \x in {1,...,3}
        \foreach \y in {1,...,3}
            \foreach \z in {1,...,3}
            {
                \addplot3 [no marks, dashed, \gridcolor] coordinates {(\x,0,0)  (\x,0,\zmax)}; 
                \addplot3 [no marks, dashed, \gridcolor] coordinates {(0,0,\z)  (\xmax,0,\z)}; 
                \addplot3 [no marks, dashed, \gridcolor] coordinates {(0,\y,0)  (0,\y,\zmax)}; 
                \addplot3 [no marks, dashed, \gridcolor] coordinates {(0,0,\z)  (0,\ymax,\z)}; 
                \addplot3 [no marks, dashed, \gridcolor] coordinates {(\x,0,0)  (\x,\ymax,0)}; 
                \addplot3 [no marks, dashed, \gridcolor] coordinates {(0,\y,0)  (\xmax,\y,0)}; 
            }
            
    \foreach \v in {(3,0,0), (0,3,0), (0,0,3), (1,1,1)}
        {
            \addplot3 [only marks, mark size=2.2pt, blue] coordinates {\v};
        }
    \node [outer sep=1pt] (1) at (axis cs:3,0,0) {};
    \node [outer sep=1pt] (2) at (axis cs:0,3,0) {};
    \node [outer sep=1pt] (3) at (axis cs:0,0,3) {};
    \node [outer sep=1pt] (4) at (axis cs:1,1,1) {};
    
    \draw [fwdrxn, blue, transform canvas={yshift=0pt}] (1)--(4) node [midway, below right=-4pt] {\ratecnst{$\kk_{14}$\!\!}}; 
    \draw [fwdrxn, blue, transform canvas={yshift=0pt}] (4)--(3) node [midway, below right] {\ratecnst{$\kk_{43}$}};
    \draw [fwdrxn, blue, transform canvas={yshift=0pt}] (3)--(2) node [midway, above right] {\ratecnst{\!\!$\kk_{32}$}};
    \draw [fwdrxn, blue, transform canvas={yshift=0pt}] (3)--(1) node [midway, above left] {\ratecnst{$\kk_{31}$\!\!}};
    \draw [fwdrxn, blue, transform canvas={yshift=0pt}] (2)--(1) node [midway, below] {\ratecnst{$\kk_{21}$\!\!}};
    \end{axis}
\end{tikzpicture}
    \caption{} 
    \label{fig:introEx-network-EEG} 
\end{subfigure} 
\hspace{0.5cm}
\begin{subfigure}[b]{0.65\textwidth}
\centering 
    \includegraphics[width=4in]{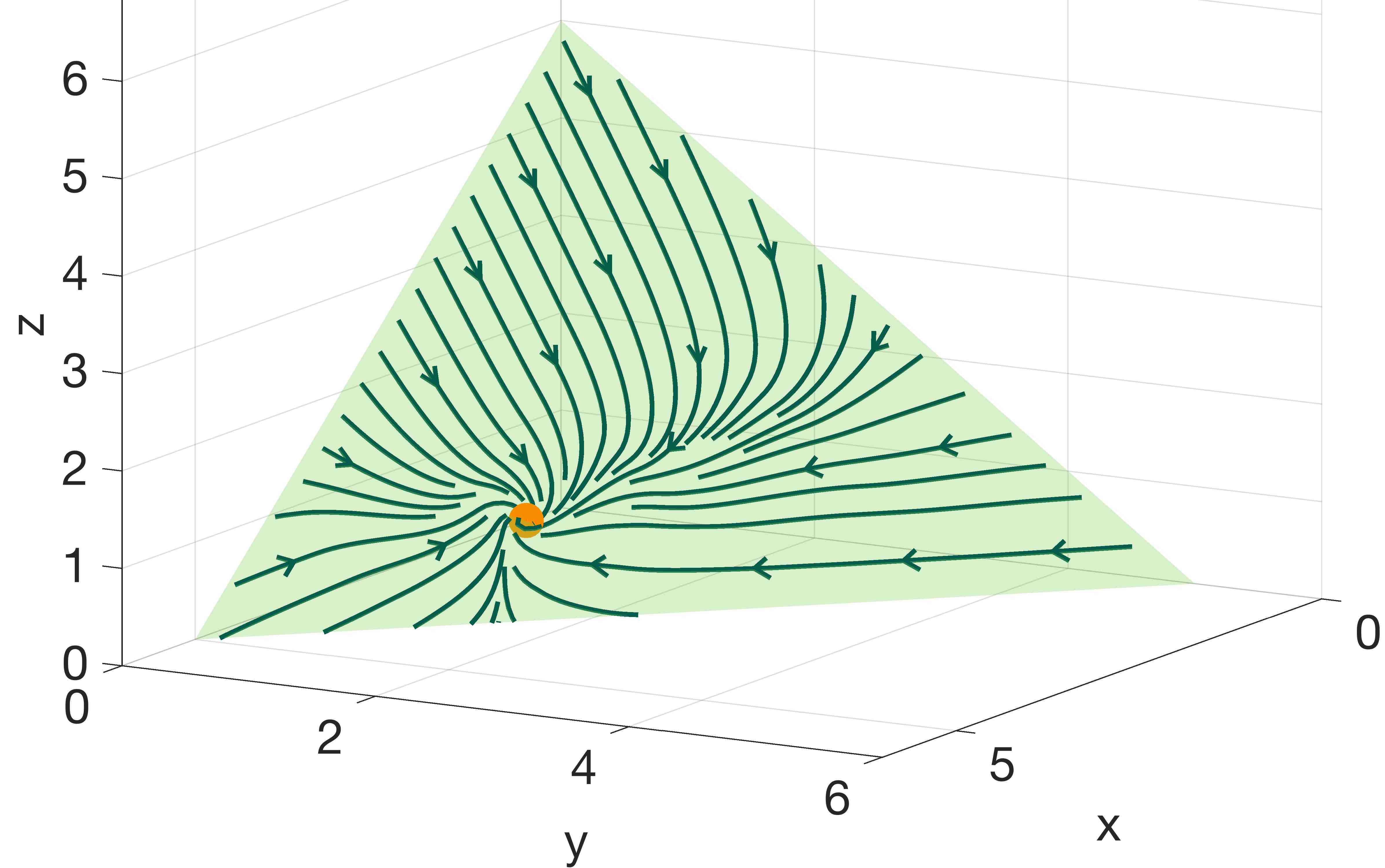}
    \caption{}
    \label{fig:introEx-traj}
\end{subfigure}
\caption{(a) A list of reactions that make up (b) the reaction network $G$. (c) The dynamics of the mass-action system $(G,\vv\kk)$ confined to a compatibility class. }
\label{fig:introEx}
\end{figure}
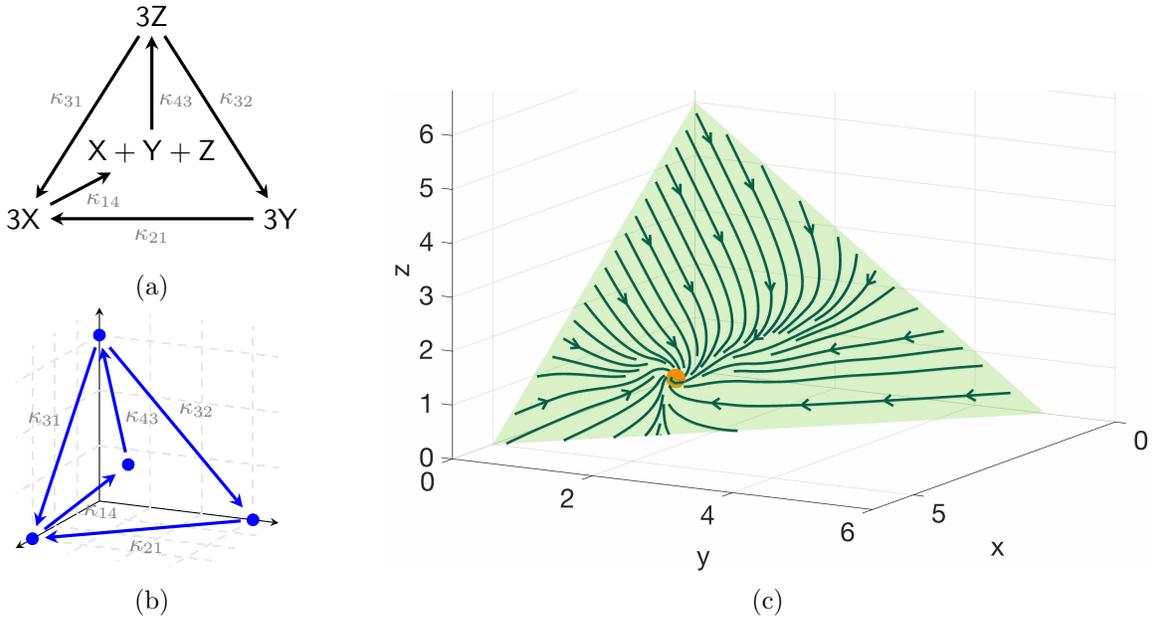

\begin{ex}
\label{ex:intro}
    Consider the five reactions listed in  \Cref{fig:introEx-network-plain}. They generate the weakly reversible reaction network $G$ in \Cref{fig:introEx-network-EEG} embedded in $\rr^3$. The associated system for any choice of $\kk_{ij} > 0$ is given by  
    \eq{ 
        \dot{x} &= -2 \kk_{14} x^3 - \hphantom{3}\kk_{43} xyz + 3 \kk_{31} z^3 + 3 \kk_{21} y^3 \\ 
        \dot{y} &= \hphantom{-2}\kk_{14} x^3 - \hphantom{3}\kk_{43} xyz + 3 \kk_{32} z^3 - 3 \kk_{21} y^3 \\ 
        \dot{z} &= \hphantom{-2}\kk_{14} x^3 +2 \kk_{43} xyz - 3 \kk_{31} z^3  - 3 \kk_{32} z^3.
    }
    The dynamics is constrained to 2-dimensional compatibility classes, an example of which is shown in \Cref{fig:introEx-traj}. Note that the stoichiometric subspace and the compatibility classes are parallel to the affine span of the network $G$.
\end{ex}

\subsection{Complex-balanced systems} 
\label{sec:CB} 

The class of complex-balanced systems was defined to generalize mass-action systems at thermodynamic equilibrium, the so-called detailed-balanced systems. As such, they enjoy most of the algebraic and dynamical properties of detailed-balanced systems.

\begin{defn}
\label{def:CB} 
A mass-action system $(G,\vv\kk)$ is said to be  \df{complex-balanced} if there exists a positive steady state $\xx \in \rrpp^n$ such that for every vertex $\yy_i \in V$, the following equality holds:
\eqn{ \label{eq:CB}
    \sum_{\yy_j \in V} \kk_{ij} \xx^{\yy_i} = \sum_{\yy_j \in V} \kk_{ji} \xx^{\yy_j}.
}
\end{defn}

A complex-balanced system is defined to be a $(G,\vv\kk)$ admitting a steady state satisfying \eqref{eq:CB}; nonetheless if a mass-action system has one complex-balanced steady state, then \emph{all} of its positive steady states are complex-balanced~\cite{HornJackson1972}. This justifies calling $(G,\vv\kk)$ complex-balanced.

Dynamically, complex-balanced systems are remarkably stable. If $\xx^*$ is a complex-balanced steady state, the function
\eq{ 
    V(\xx) = \sum_{i=1}^n x_i \ln (x_i - x_i^* -1) 
}
serves as a Lyapunov function on all of $\rrpp^n$. The unique minimum of $V$ within each compatibility class is a complex-balanced steady state, which is conjectured to be globally stable within its compatibility class~\cite{Horn1974_GAC, CraciunNazarovPantea2013GAC}. The latter statement, known as the \emph{Global Attractor Conjecture} in reaction network theory, is proved only for several cases; for example, strongly connected networks~\cite{Anderson2011GAC, BorosHofbauer2019}, strongly endotactic networks~\cite{GopalkrishnanMillerShiu2014}, networks in $\rr^2$~\cite{CraciunNazarovPantea2013GAC}, and networks with three-dimensional stoichiometric subspaces~\cite{Pantea2012}.

A complex-balanced steady state is not only locally asymptotically stable within its compatibility class; it is linearly stable~\cite{Johnston2008, Johnston2011_thesis, BorosMullerRegensburger2019}. The stable manifold at the steady state $\xx^*$ coincides with the compatibility class $(\xx^*+S)_>$, while the unstable manifold is trivial, and the centre subspace is parallel to $\diag(\xx^*)S^\perp$~\cite{Johnston2008, Johnston2011_thesis}, which is linearly independent of $S$. Consequently, the dynamics of a complex-balanced system on a compatibility class is  diffeomorphic to a lower-dimensional system.

\begin{thm}[\cite{Johnston2008, Johnston2011_thesis}]
\label{thm:JohnstonLinearizedStability}
    Consider $(G,\vv\kk)$ and its associated system $\dot{\xx} = \vv f(\xx; \vv\kk)$, where $s = \dim S$. Let $\xx^* > \vv 0$ be a complex-balanced steady state in the compatibility class $U = (\xx_0 + S)_>$. Then $U$ is forward-invariant, and coincides with the stable manifold of the system. The system $\left.\dot{\xx} \right|_{U}$ is diffeomorphic to a $s$-dimensional system 
        \eq{
            \dot{\vv z} = \vv F(\vv z; \vv\kk, \vv x_0),
        } 
        which has a unique linearly stable steady state. 
\end{thm}

Algebraically, complex-balanced systems have toric structures. Its set of positive steady states is a toric variety, and after a suitable change of variables, so is the set of rate constants $\mc K(G)$ for which the system is complex-balanced. The \df{toric locus} $\mc K(G)$ is in general a set of measure zero; more precisely, $\mc K(G)$ is cut out by $\delta$ algebraic equations, where 
\eq{
    \delta = |V| - \ell - \dim S
}
is the \df{deficiency} of the network $G$ with $\ell$ connected components and stoichiometric subspace $S$~\cite{feinberg1972complex, horn1972necessary, ToricDynSys2009}.

\begin{ex}
\label{ex:intro-deficiency}
\begin{figure}[h!]
\centering 
    \includegraphics[width=4in]{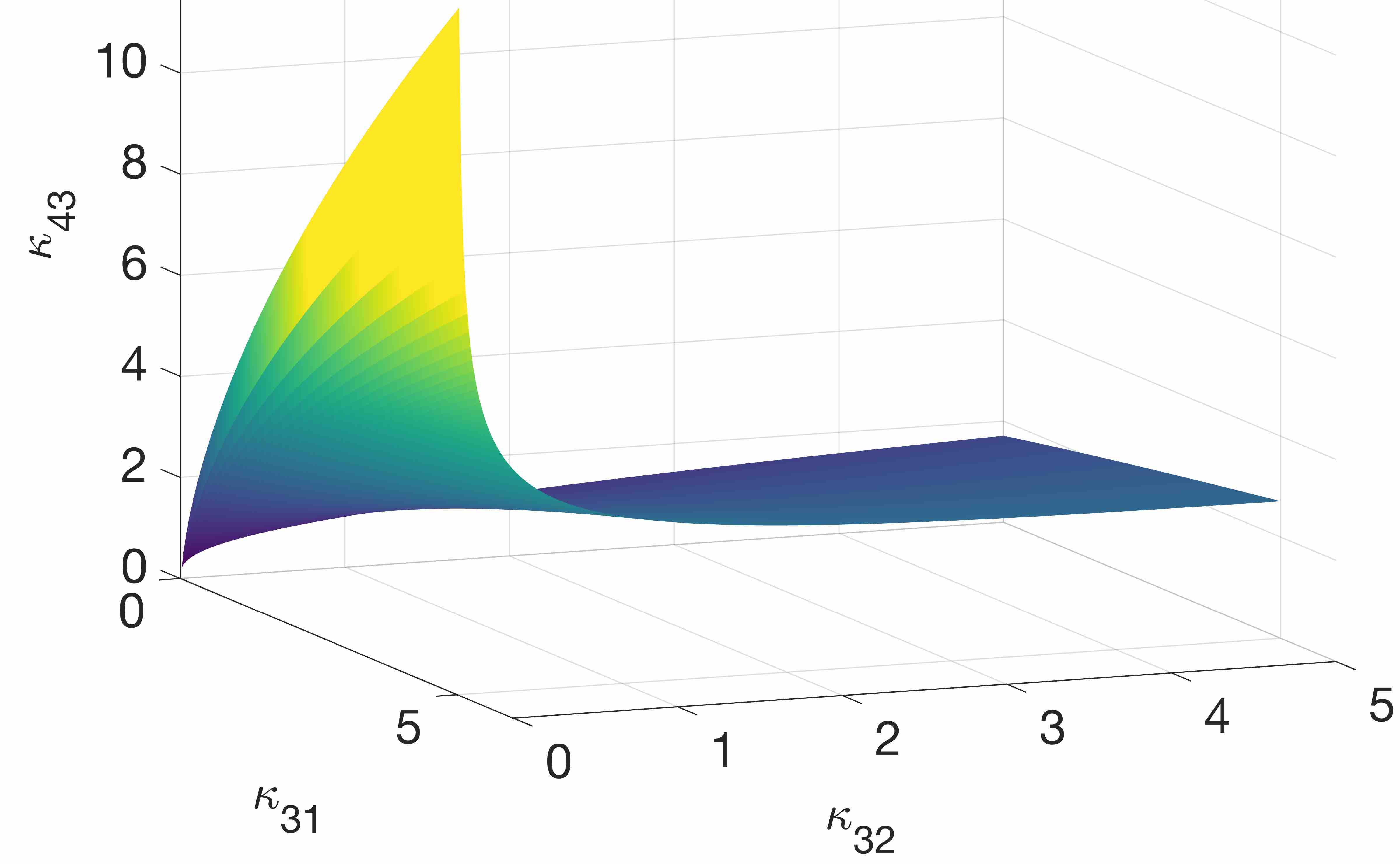}
\caption{The toric locus $\mc K(G)$ of the  deficiency one network in \Cref{fig:introEx}, with $\kk_{21} = \kk_{14}=1$. The toric locus is a measure zero set.}
\label{fig:introEx-param}
\end{figure}

    The deficiency of the network $G$ in \Cref{fig:introEx} is $\delta = 1$. Since $G$ is weakly reversible, there is one algebraic condition on the rate constants $\vv\kk$ that is necessary and sufficient for $(G,\vv\kk)$ to be complex-balanced, namely $K_1K_2K_3 = K_4^3$, where 
    \eq{ 
        K_1 &= \kk_{43}\kk_{21}(\kk_{31}+\kk_{32})\\
        K_2 &= \kk_{14}\kk_{43}\kk_{32}\\ 
        K_3 &= \kk_{21}\kk_{14}\kk_{43}\\ 
        K_4 &= (\kk_{31}+\kk_{32})\kk_{21}\kk_{14}
    } 
    are given by the Matrix-Tree Theorem~\cite{ToricDynSys2009}. After simplification, the condition becomes $\kk_{43}^3\kk_{32} = \kk_{21}\kk_{14}(\kk_{31}+\kk_{32})^2$, a slice of which is plotted in \Cref{fig:introEx-param}. As the figure demonstrates, for arbitrarily chosen $\vv\kk$, the system is not complex-balanced.
    In the case when the condition on the rate constants is met, any positive steady state $(x,y,z)$ must satisfied \eqref{eq:CB}: 
    \eq{ 
        \kk_{14} x^3 &= \kk_{21} y^3 + \kk_{31} z^3 \\  
        \kk_{21} y^2   &= \kk_{32} z^3 \\
        (\kk_{31}+\kk_{32})  z^3 &= \kk_{43} xyz \\
        \kk_{43} xyz &= \kk_{14} x^3 ,
    }
    which describe the balancing of flows through each vertex. 
\end{ex}

The only networks that are complex-balanced for all $\vv\kk \in \rrpp^r$ are weakly reversible and have deficiency zero~\cite{horn1972necessary}. Therefore when $\delta \geq 1$, perturbations in the rate constants leave the algebraic conditions unsatisfied; complex-balancing is \emph{not} robust in general.  

For networks with $\delta \geq 1$, it is possible that although $(G,\vv\kk)$ itself is not complex-balanced, it is \emph{dynamically equivalent} to a complex-balanced system~\cite{CraciunJinYu2019}, i.e., there exist a different network $G'$ and a vector of rate constants $\vv\kk'$ such that $(G',\vv\kk')$ is complex-balanced and shares the same associated ODEs as $(G,\vv\kk)$. However, not every network is afforded this flexibility. For example, the network
\begin{center} 
\begin{tikzpicture}[scale=1.5]
    \node at (0,1.1) {}; 
    \node (0) at (0,0) {$\sf{0}$};
    \node (x) at (1.25,0)   {$\sf{X}$};
    \node (xy) at (1.25,1)  {$\sf{X}+\sf{Y}$};
    \node (y) at (0,1) {$\sf{Y}$};
    \draw [fwdrxn] (0) -- (x); 
    \draw [fwdrxn] (x) -- (xy); 
    \draw [fwdrxn] (xy) -- (y); 
    \draw [fwdrxn] (y) -- (0); 
\end{tikzpicture}
\end{center}
does not admit a different weakly reversible structure. Thus, for a given choice of $\vv\kk > 0$, either $(G,\vv\kk)$ itself is complex-balanced, or the same associated dynamics is never generated by a complex-balanced system.

\subsection{Permanence} 
\label{sec:permanence} 

The notion of permanence is prominent in interaction network models, especially for long term survival dynamics in population models. Informally, a system is permanent if there is a globally attracting compact set.  
The formal definition of permanence is as follow.

\begin{defn}
\label{def:permanence}
    A system of ODEs $\dot{\xx} = \vv f(\xx, t)$ on $U \subseteq \rr^n$ is said to be \df{permanent on $U$} if there exists a compact set $K \subseteq U$ such that any solution $\xx(t)$ with initial condition in $U$ is eventually in $K$, i.e., there exists $t_0 \geq 0$ such that $\xx(t) \in K$ for all $t \geq t_0$.
\end{defn}

A mass-action system $(G,\vv\kk)$ is said to be  \df{permanent} if the associated system $\dot{\xx} = \vv f(\xx)$ is permanent on every compatibility class. Permanence for variable-$\vv\kk$ mass-action system is defined in a similar manner, where the aforementioned set $K$ on a compatibility class is globally attracting for \emph{any} coefficient functions satisfying $\eps \leq \kk_{ij}(t) \leq \eps^{-1}$. 

For complex-balanced systems, 
permanence is sufficient to prove the Global Attractor Conjecture~\cite{CraciunNazarovPantea2013GAC}. 
It is believed that permanence is not only a property of complex-balanced systems, but of the more general weakly reversible systems.

\begin{conj}[Permanence Conjecture~\cite{CraciunNazarovPantea2013GAC}]
    Any weakly reversible mass-action system is permanent. 
\end{conj}

\begin{conj}[Permanence Conjecture for variable-$\vv\kk$ systems~\cite{CraciunNazarovPantea2013GAC}]
    Any weakly reversible variable-$\vv\kk$ mass-action system is permanent. 
\end{conj}

In this work, we define the notion of robust permanence for a family of systems. We prove that robust permanence implies the global stability of any system that arises from perturbing the rate constants of a complex-balanced system.

\begin{defn}
\label{def:robustperm}
    A system of ODEs  
    $\dot{\xx} = \vv f(\xx; \vv\kk^*)$
    on $U \subseteq \rr^n$ is said to be \df{robustly permanent on $U$ with respect to $\vv\kk^*$} if there exist a compact set $K \subseteq U$ and $\eps > 0$ such that whenever $\vv\kk \in B(\vv\kk^*,\eps)$, any solution to $\dot\xx = \vv f(\xx; \vv\kk)$ with initial condition in $U$ is eventually in $K$. 
\end{defn}

Note that robust permanence with respect to $\vv\kk^*$ implies that $\dot{\xx} = \vv f(\xx; \vv \kk)$ is permanent for all $\vv\kk \in B(\vv\kk^*, \eps)$. Moreover, the attracting compact set $K$ is common to all $\vv\kk \in B(\vv\kk^*, \eps)$.

We say a mass-action system $(G, \vv\kk^*)$ is \df{robustly permanent with respect to $\vv\kk^*$} if the system $\dot{\xx} = \vv f(\xx; \vv\kk^*)$ is robustly permanent on every compatibility class. In this case, the same compact set $K$ is globally attracting on the compatibility class for all $\vv\kk \in B(\vv\kk^*,\eps)$.

Robust permanence for complex-balanced systems follows from the Permanence Conjecture for variable-$\vv\kk$ systems.  To date, the conjecture has been proven for systems on $\rrpp^2$~\cite{CraciunNazarovPantea2013GAC}, networks with one connected component~\cite{Anderson2011GAC, BorosHofbauer2019}, and strongly endotactic networks~\cite{GopalkrishnanMillerShiu2014}.

\begin{prop}
\label{prop:robustpermanence}
    Assume the Permanence Conjecture for variable-$\vv\kk$ systems. Then a complex-balanced system $(G,\vv\kk^*)$ is robustly permanent with respect to $\vv\kk^*$. In particular, any complex-balanced steady state of $(G,\vv\kk^*)$ is globally attracting within its compatibility class.
\end{prop}
\begin{proof}
    Let $\eps >0$ be chosen such that $\eps < \kk^*_{ij} < \eps^{-1}$ for every $(i,j) \in E$. Choose $\delta > 0$ such that the ball $B(\vv\kk^*,\delta)$ lies within the box $[\eps, \eps^{-1}]^r$. 
    Since complex-balanced systems are weakly reversible, the Permanence Conjecture for variable-$\vv\kk$ systems provides a globally attracting compact set, one for each compatibility class, for $\dot{\xx} = \vv f(\xx; \vv\kk(t))$ with any $\eps \leq \kk_{ij}(t) \leq \eps^{-1}$. Robust permanence with respect to $\vv\kk^*$ follows by considering constant coefficient functions $\vv\kk(t) \equiv \vv\kk$ for any $\vv\kk \in B(\vv\kk^*,\delta)$. It is well-known that permanence implies global stability for complex-balanced systems. 
\end{proof}

\section{Main result}
\label{sec:main}

In this section, we prove that if a complex-balanced system $(G,\vv\kk^*)$ is robustly permanent, then any mass-action system $(G,\vv\kk)$, with $\vv\kk$ is sufficiently close to $\vv\kk^*$, is globally stable, even though the perturbed system $(G,\vv\kk)$ is in general \emph{not} complex-balanced. 

We use the following result from \cite{SmithWaltman1999}. Consider a system of ODEs (with parameters $\vv\kk$) 
    \eqn{
    \label{eq:SmithWaltman}
        \frac{d\vv z}{dt} = \vv F(\vv z; \vv \kk),
    }
where $\vv F: U \times \Lambda \to \rr^n$ with $U \subseteq \rr^n$ and $\Lambda \subseteq \rr^r$. Suppose the Jacobian matrix $\grad_{\vv z}\vv F(\vv z; \vv \kk)$ is continuous on $U \times \Lambda$. Assume that solutions of the initial value problems are unique and remain in $U$ for all $t \geq 0$ and $\vv \kk \in \Lambda$. Write $\vv z(t; \vv z_0, \vv \kk)$ for the solution of \eqref{eq:SmithWaltman} with initial condition $\vv z(0) = \vv z_0$.

\newpage 
\begin{thm}[\cite{SmithWaltman1999}]
\label{thm:SmithWaltman}
For the system \eqref{eq:SmithWaltman},  assume that  
\begin{enumerate}[label={\textrm{\roman*)\,\,}}]
    \item $\vv F(\vv z^*; \vv \kk^*) = \vv 0$, where $\vv z^*$ is an interior point of $U$;
    \item all eigenvalues of $\grad_\vv z\vv F(\vv z^*; \vv \kk^*)$ have negative real part; and
    \item $\vv z^*$ is globally attracting for solutions of \eqref{eq:SmithWaltman} when $\vv \kk = \vv \kk^*$.
\end{enumerate}
Suppose that there exists a compact set $K \subseteq U$ such that for each $\vv \kk \in \Lambda$ and each $\vv z_0 \in U$, we have $\vv z(t; \vv z_0, \vv \kk) \in K$ for all large $t$. Then there exist $\eps > 0$ and a unique point $\hat{\vv z}(\vv \kk) \in U$ for all $\vv\kk \in B(\vv\kk^*,\eps)$  such that $\vv F(\hat{\vv z}(\vv \kk); \vv\kk) = \vv 0$ and $ \vv z(t; \vv z_0, \vv \kk) \xrightarrow{t\to\infty} \hat{\vv z}(\vv \kk)$ for all $\vv z_0 \in U$. 
\end{thm}

Moreover, the function $\vv\kk \mapsto \hat{\vv z}(\vv \kk)$ is continuous on a neighbourhood of $\vv\kk^*$~\cite{SmithWaltman1999}; thus $\hat{\vv z}(\vv\kk)$ is an interior point of $U$.

Our main result is the following.
\begin{thm}
\label{thm:main-k}
    Let $(G,\vv\kk^*)$ be a complex-balanced system that is robustly permanent with respect to $\vv\kk^*$. Then on every compatibility class $U$, there exists $\eps > 0$ such that for every $\vv \kk \in B(\vv\kk^*,\eps)$, the mass-action system $(G,\vv\kk)$ has a unique globally attracting point within $U$.
\end{thm}
\begin{rmk}
    Note that from the definition of robust permanence, $\eps$ depends on $\vv\kk^*$. See \Cref{ex:ms} and \Cref{fig:ms-bifurcation}. 
\end{rmk}
\begin{proof}
    Let 
    \eqn{ \label{eq:pf-mas}
        \dot{\xx} = \vv f(\xx; \vv\kk)
    }
    be the system of ODEs associated to $(G,\vv\kk)$, which is the complex-balanced when $\vv\kk = \vv\kk^*$. Fix a compatibility class $\tilde U = (\vv x_0 + S)_>$, and let $\xx^*$ be the unique stable steady state in $\tilde U$. 
    Let $\tilde K \subseteq \tilde U$ be the globally attracting compact set within $\tilde U$ for \eqref{eq:pf-mas} for any $\vv\kk \in B(\vv\kk^*,\eps)$. Let $\Lambda = B(\vv\kk^*,\eps)$. 
    
    By \Cref{thm:JohnstonLinearizedStability}, the dynamics of $\dot{\xx} = \vv f(\xx; \vv\kk^*)$  on $\tilde U$ is diffeomorphic (via a diffeomorphism $\Phi$) to the lower dimensional system 
    \eqn{ \label{eq:pf-reduced}
        \dot{\vv z} = \vv F(\vv z; \vv\kk^*, \vv x_0), 
    }
    whose domain is $U = \Phi(\tilde U)$. Correspondingly, $\vv z^* = \Phi(\vv x^*)$ is the globally attracting steady state in $U$ and $K = \Phi(\tilde K)$ a globally attracting compact set on $U$. Moreover, robust permanence means that $K$ is globally attracting for  $\dot{\vv z} = \vv F(\vv z; \vv\kk, \vv x_0)$ for any $\vv\kk \in B(\vv\kk^*,\eps)$. Finally, the right-hand side of \eqref{eq:pf-mas} is differentiable in $\xx$ and linear in $\vv\kk$, so its Jacobian matrix is continuous in $\xx$ and $\vv \kk$. Thus the Jacobian matrix of \eqref{eq:pf-reduced} is continuous on $U \times \Lambda$. The theorem follows from \Cref{thm:SmithWaltman} and $\Phi^{-1}$. 
\end{proof}

As pointed out in \Cref{prop:robustpermanence},  robust permanence follows from the Permanence Conjecture for variable-$\vv\kk$ systems, which has implication for the Global Attractor Conjecture and has been proven for several special cases. For completeness, we state the following corollaries.

\begin{cor}
\label{cor:perm_conj}
    Assume the Permanence Conjecture for variable-$\vv\kk$ mass-action systems. Let $(G,\vv\kk^*)$ be complex-balanced. Then on every compatibility class $U$, there exists $\eps > 0$ such that for every $\vv \kk \in B(\vv\kk^*,\eps)$, the mass-action system $(G,\vv\kk)$ has a unique globally attracting point within $U$.
\end{cor}

\begin{cor}
\label{cor:specialcases}
    Let $G$ be weakly reversible reaction network that either has one connected component or is strongly endotactic or is in $\rr^2$. 
    Let $(G,\vv\kk^*)$ be complex-balanced. Then on every compatibility class $U$, there exists $\eps > 0$ such that for every $\vv \kk \in B(\vv\kk^*,\eps)$, the mass-action system $(G,\vv\kk)$ has a unique globally attracting point within $U$.
\end{cor}

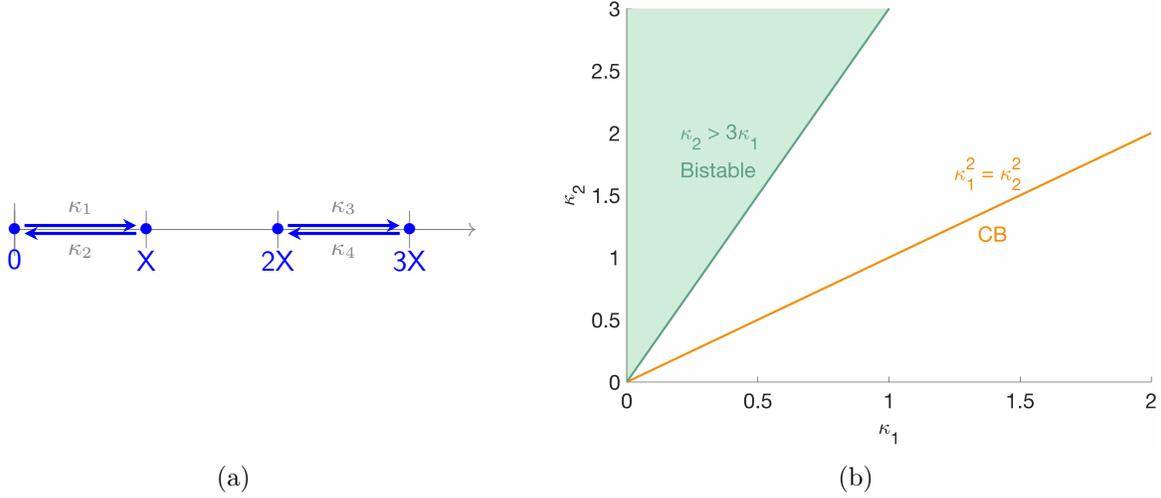
\begin{figure}[h!]
\centering
\begin{subfigure}[b]{0.38\textwidth}
\centering 
    \begin{tikzpicture}[xscale=1.75]
    \draw [step=1, gray, very thin] (0,-0.25) grid (3.5,0.25);
        \draw [ ->, gray] (0,0)--(3.5,0);
        \draw [  gray] (0,-0.35)--(0,0.35);

    \node [blue,outer sep=0pt, inner sep=1pt] (0) at (0,0) {$\bullet$}; 
    \node [blue,outer sep=0pt, inner sep=1pt] (1) at (1,0) {$\bullet$}; 
    \node [blue,outer sep=0pt, inner sep=1pt] (2) at (2,0) {$\bullet$}; 
    \node [blue,outer sep=0pt, inner sep=1pt] (3) at (3,0) {$\bullet$}; 
    
    \node [blue] at (0) [below=4pt] {$\sf{0}$};
    \node [blue] at (1) [below=4pt] {$\sf{X}$};
    \node [blue] at (2) [below=4pt] {$\sf{2X}$};
    \node [blue] at (3) [below=4pt] {$\sf{3X}$};
    
    \draw[blue, fwdrxn, transform canvas={yshift=1.5pt}] (0)--(1) node [midway, above] {\ratecnst{$\kk_1$}};
    \draw[blue, fwdrxn, transform canvas={yshift=-1.5pt}] (1)--(0) node [midway, below] {\ratecnst{$\kk_2$}};
    
    \draw[blue, fwdrxn, transform canvas={yshift=1.5pt}] (2)--(3) node [midway, above] {\ratecnst{$\kk_3$}};
    \draw[blue, fwdrxn, transform canvas={yshift=-1.5pt}] (3)--(2) node [midway, below] {\ratecnst{$\kk_4$}};
    
\end{tikzpicture}
\vspace{1.75cm} 
    \caption{}
    \label{fig:ms-network}
\end{subfigure} 
\hspace{0.75cm}
\begin{subfigure}[b]{0.5\textwidth}
\centering 
    \includegraphics[width=3.15in]{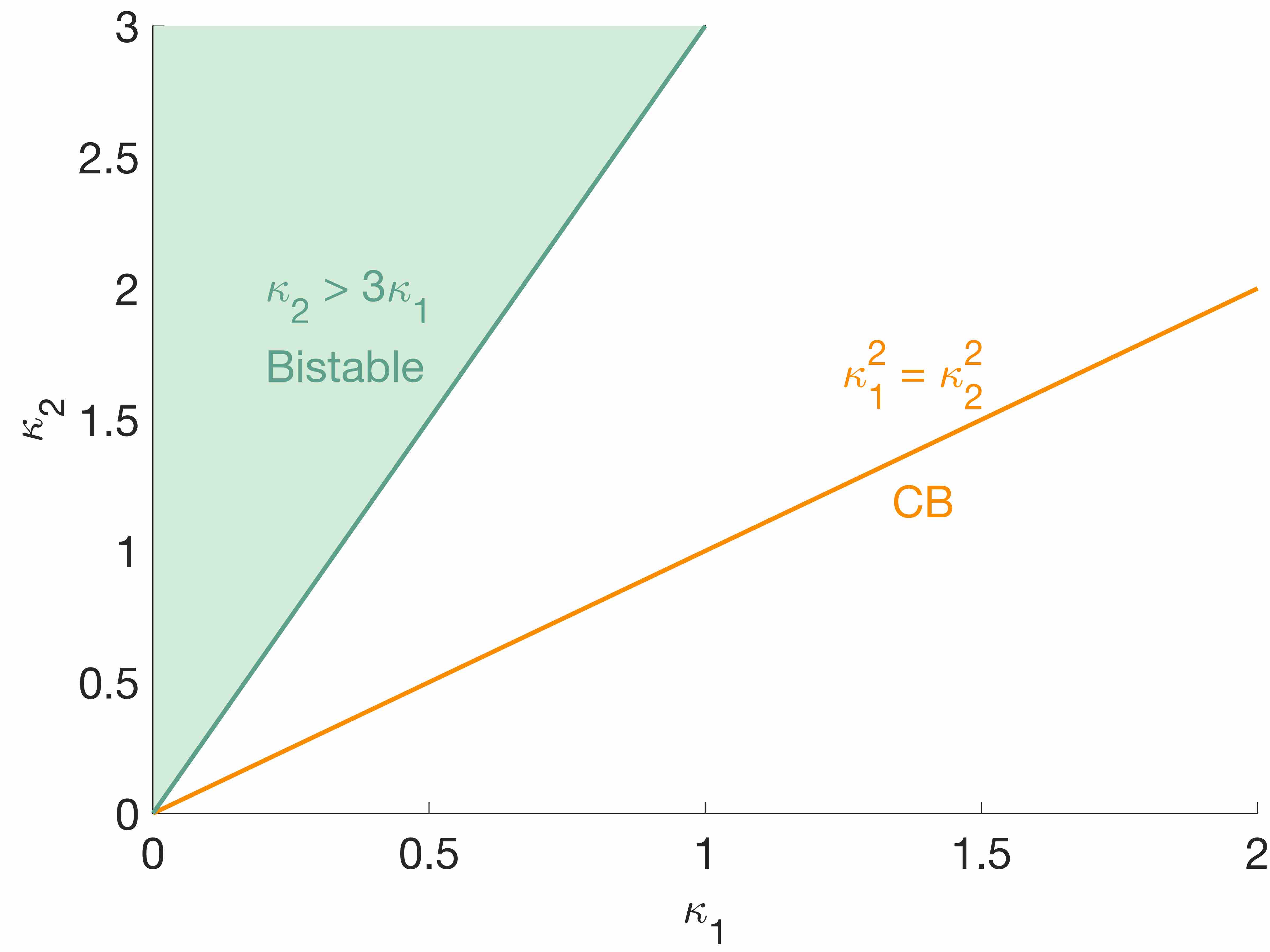}
    \caption{}
    \label{fig:ms-bifurcation}
\end{subfigure} 
\caption{(a) A network with the capacity for multiple steady states. (b) Bifurcation diagram when $\kk_1 = \kk_4$ and $\kk_2 =\kk_3$. The mass-action system is complex-balanced when $\kk_1 =\kk_2$ (orange line); it is bistable when $\kk_2 > 3\kk_1$ (green region). As $\min \kk_i \to 0$, the $\eps$-ball of \Cref{thm:main-k} necessarily gets smaller. }
\label{fig:ms}
\end{figure}

\subsection{Examples}
\label{sec:examples}

\begin{ex}
\label{ex:ms} 
    Consider the network in \Cref{fig:ms-network}, which has the capacity for multiple positive steady states. This example illustrates that \Cref{thm:main-k} applies even when the system can exhibit interesting dynamics, and that $\eps$ from the theorem depends on the choice of $\vv\kk^*$. 
    
    For clarity of discussion, we assume that $\kk_1 = \kk_4$ and $\kk_2 = \kk_3$ with no great loss of generality. There is a steady state at $x^* =1 $, whose eigenvalue is $\kk_2  - 3 \kk_1$. Clearly for $\kk_2 < 3\kk_1$, the steady state is asymptotically stable, hence globally stable. It is complex-balanced if and only if $\kk_1 = \kk_2$. It is not difficult to show that the remaining steady states are given by 
    \eq{ 
        x^* = \frac{\kk_2 -\kk_1  \pm \sqrt{(\kk_2-\kk_1)^2 -4\kk_1^2} }{2\kk_1}, 
    }
    which are real (and positive) if and only if $\kk_2 \geq 3\kk_1$. Indeed, the dynamics of this mass-action systems can be summarized by  \Cref{fig:ms-bifurcation}: global stability when $\kk_2 < 3\kk_1$; bistability when $\kk_2 > 3\kk_1$, with the transition across $\kk_2 = 3\kk_1$ being a pitchfork bifurcation. 
    
    Furthermore, we can see from \Cref{fig:ms-bifurcation} that the upper bound on $\eps$ from \Cref{thm:main-k} depends on the initial choice of $\vv\kk^*$ that makes the system complex-balanced. In this case, $\vv\kk^*$ satisfies $\kk_1^2 = \kk_2^2$; more generally, $\kk_1\kk_4 = \kk_2\kk_3$. In particular, $\eps$ necessarily has to be small when $\kk_i$ is small. 
\end{ex}

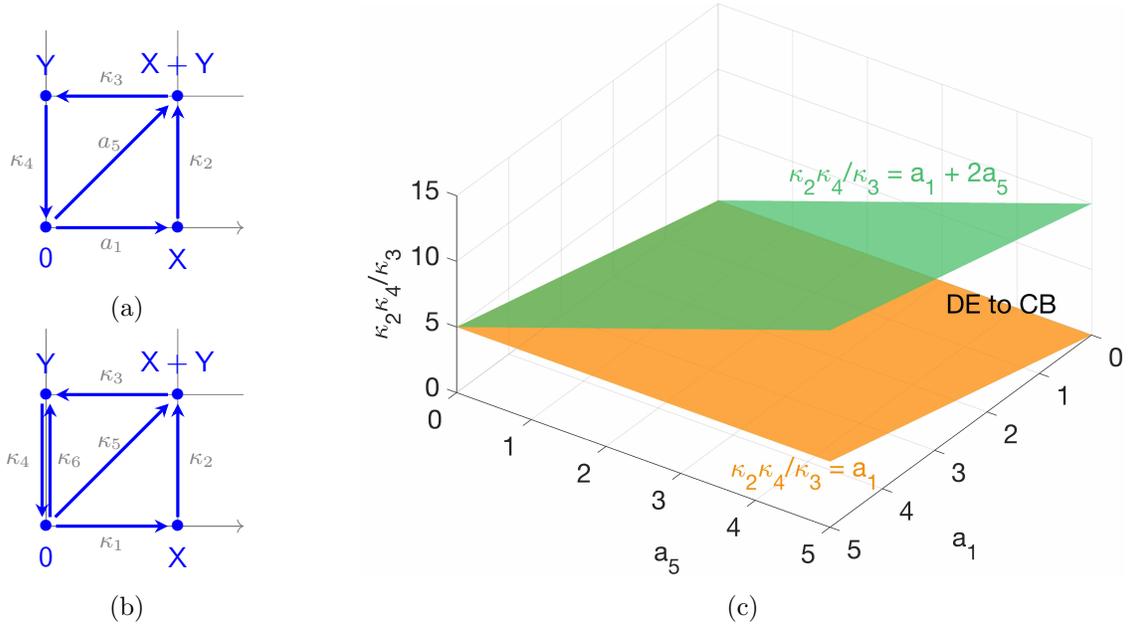
\begin{figure}[h!]
\centering
\begin{subfigure}[b]{0.25\textwidth}
\centering 
    \begin{tikzpicture}[scale=1.75]
    \draw [step=1, gray, very thin] (0,0) grid (1.5,1.5);
        \draw [ ->, gray] (0,0)--(1.5,0);
        \draw [  gray] (0,0)--(0,1.5);

    \node [blue,outer sep=0pt, inner sep=1pt] (0) at (0,0) {$\bullet$}; 
    \node [blue,outer sep=0pt, inner sep=1pt] (1) at (1,0) {$\bullet$}; 
    \node [blue,outer sep=0pt, inner sep=1pt] (2) at (1,1) {$\bullet$}; 
    \node [blue,outer sep=0pt, inner sep=1pt] (3) at (0,1) {$\bullet$}; 
    
    \node [blue] at (0) [below=4pt] {$\sf{0}$};
    \node [blue] at (1) [below=4pt] {$\sf{X}$};
    \node [blue] at (2) [above=4pt] {$\sf{X+Y}$};
    \node [blue] at (3) [above=4pt] {$\sf{Y}$};
    
    \draw[blue, fwdrxn] (0)--(1) node [midway, below] {\ratecnst{$a_1$}};
    \draw[blue, fwdrxn] (1)--(2) node [midway, right] {\ratecnst{$\kk_2$}};
    \draw[blue, fwdrxn] (2)--(3) node [midway, above] {\ratecnst{$\kk_3$}};
    \draw[blue, fwdrxn] (3)--(0) node [midway, left] {\ratecnst{$\kk_4$}};
    \draw[blue, fwdrxn] (0)--(2) node [midway, above] {\ratecnst{$a_5$\,}};
    
\end{tikzpicture}
    \caption{} 
    \label{fig:square-diag-Ex-network} 
    
    \begin{tikzpicture}[scale=1.75]
    \draw [step=1, gray, very thin] (0,0) grid (1.5,1.5);
        \draw [ ->, gray] (0,0)--(1.5,0);
        \draw [  gray] (0,0)--(0,1.5);

    \node [blue,outer sep=0pt, inner sep=1pt] (0) at (0,0) {$\bullet$}; 
    \node [blue,outer sep=0pt, inner sep=1pt] (1) at (1,0) {$\bullet$}; 
    \node [blue,outer sep=0pt, inner sep=1pt] (2) at (1,1) {$\bullet$}; 
    \node [blue,outer sep=0pt, inner sep=1pt] (3) at (0,1) {$\bullet$}; 
    
    \node [blue] at (0) [below=4pt] {$\sf{0}$};
    \node [blue] at (1) [below=4pt] {$\sf{X}$};
    \node [blue] at (2) [above=4pt] {$\sf{X+Y}$};
    \node [blue] at (3) [above=4pt] {$\sf{Y}$};
    
    \draw[blue, fwdrxn] (0)--(1) node [midway, below] {\ratecnst{$\kk_1$}};
    \draw[blue, fwdrxn] (1)--(2) node [midway, right] {\ratecnst{$\kk_2$}};
    \draw[blue, fwdrxn] (2)--(3) node [midway, above] {\ratecnst{$\kk_3$}};
    \draw[blue, fwdrxn, transform canvas={xshift=-1.5pt}] (3)--(0) node [midway, left] {\ratecnst{$\kk_4$}};
    \draw[blue, fwdrxn, transform canvas={xshift=1.5pt}] (0)--(3) node [midway, right] {\ratecnst{\!$\kk_6$}};
    \draw[blue, fwdrxn] (0)--(2) node [midway, above] {\ratecnst{$\kk_5$\,}};
    
\end{tikzpicture}
    \caption{} 
    \label{fig:square-diag-Ex-DE} 
\end{subfigure} 
\hspace{0.5cm}
\begin{subfigure}[b]{0.65\textwidth}
\centering 
    \hspace{5cm}
    \includegraphics[width=4in]{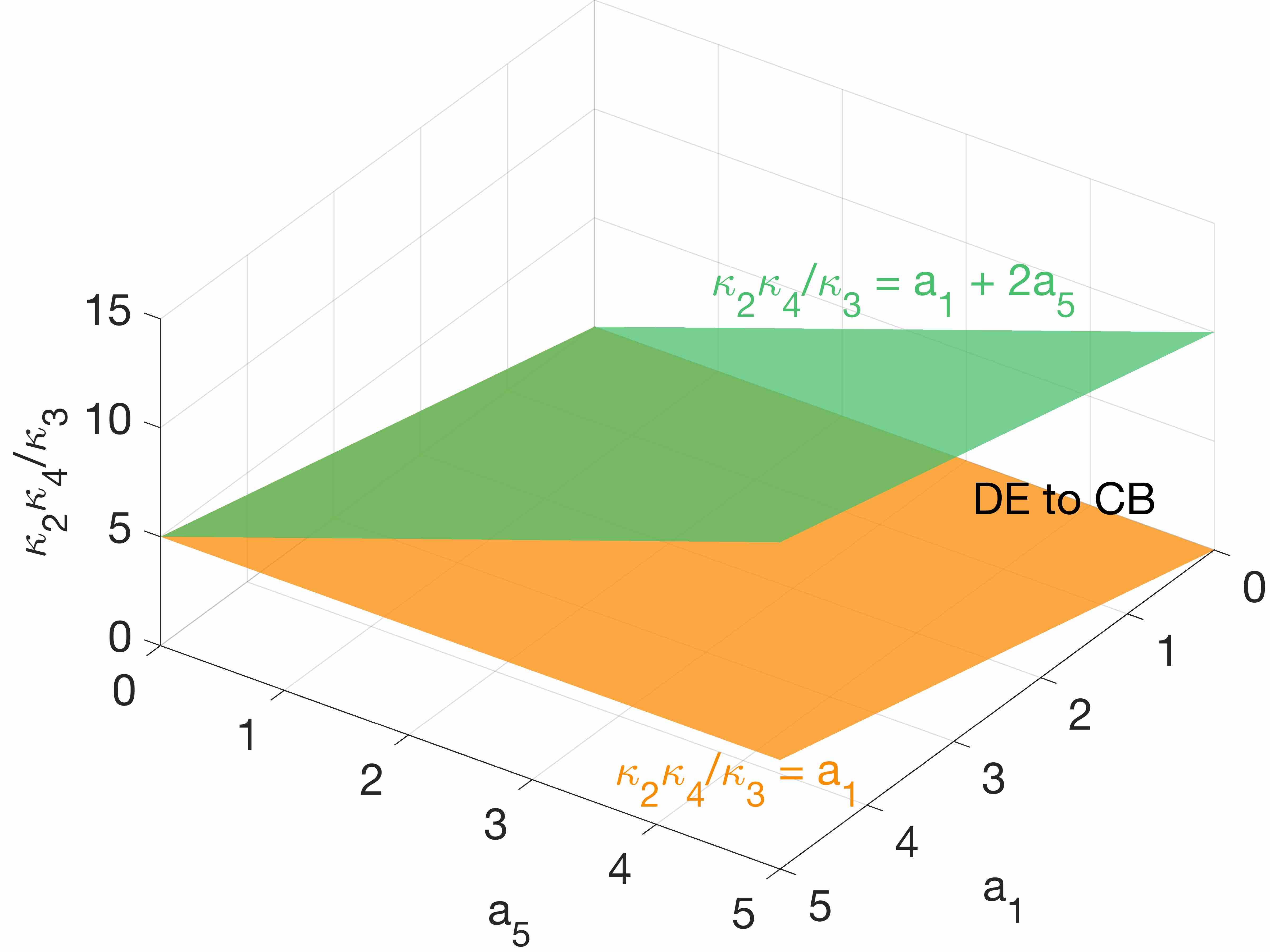}
    \caption{}
    \label{fig:square-diag-Ex-param}
\end{subfigure}
\caption{(a) A deficiency one mass-action system that is complex-balanced if and only if $\frac{\kk_2\kk_4}{\kk_3} = a_1$. It has the same dynamics as the one in (b) if $\kk_1 + \kk_5 = a_1 + a_5$ and $\kk_5 + \kk_6 = a_5$. (c) The lower surface is the toric locus of the system in (a). Dynamical equivalence to complex-balancing allows us to conclude global stability for all rate constants between the two surfaces or in a neighbourhood of the surfaces. }
\label{fig:square-diag-Ex}
\end{figure}

\begin{ex}
\label{ex:square-diag}
    The following example not only illustrates how our result applies directly, but also how we can vastly expand the parameter region for global stability by considering networks that are \emph{dynamically equivalent to complex-balancing}~\cite{CraciunJinYu2019, HaqueSatrianoSoreaYu2022, BCS22}. 
    Consider the weakly reversible network $G$ in \Cref{fig:square-diag-Ex-network}, with arbitrary rate constants $a_1$, $a_5$, $\kk_2$, $\kk_3$, $\kk_4 > 0$. The associated dynamics
    \begin{equation}\label{eq:ex-original} 
    \begin{aligned} 
        \dot{x} &= a_1 + a_5 - \kk_3 xy \\
        \dot{y} &= a_5 +\kk_2 x - \kk_4 y
    \end{aligned}
    \end{equation}
    has exactly one positive steady state,  
    which is complex-balanced when $K_1K_4 = K_2K_3$, where 
    \eq{ 
        K_1 = \kk_2\kk_3\kk_4 , \qquad & K_3 = \kk_2\kk_4(a_1+a_5) \\ 
        K_2 = a_1\kk_3\kk_4 , \qquad & K_4 = \kk_2\kk_3(a_1+a_5) ,
    }
    or after simplification,  $\kk_2\kk_4/\kk_3 = a_1 $ (lower orange surface in \Cref{fig:square-diag-Ex-param}). Clearly this does not hold for arbitrarily chosen rate constant. However, by \Cref{cor:specialcases}, we conclude that the system is globally stable for rate constants that are sufficiently close to satisfying $\kk_2\kk_4/\kk_3 = a_1 $. 

    We can expand the region for global stability using the notion of dynamical equivalence. The system \eqref{eq:ex-original} can be written as 
    \begin{equation}\label{eq:ex-DE} 
    \begin{aligned} 
        \dot{x} &= \kk_1 + \kk_5 - \kk_3 xy \\
        \dot{y} &= \kk_5 + \kk_6 +\kk_2 x - \kk_4 y, 
    \end{aligned}
    \end{equation}
    where $\kk_1 + \kk_5 = a_1 + a_5$ and $\kk_5 + \kk_6 = a_5$. The system \eqref{eq:ex-DE} is the associated system of the network in \Cref{fig:square-diag-Ex-DE}, and this system is complex-balanced if and only if 
    \eqn{\label{eq:ex-MMT} 
        \kk_2\kk_4(\kk_1+\kk_5) = \kk_1\kk_3(\kk_1+\kk_5+\kk_6).
    }
    It is not difficult to show that for any $a_1$, $a_5$, $\kk_2$, $\kk_3$, $\kk_4 > 0$ such that $a_1 <  \frac{\kk_2\kk_4}{\kk_3} < a_1 + 2a_5$, we can choose positive parameters 
    \eq{ 
        \kk_1 &= \frac{-a_5 + \sqrt{a_5^2 + 4(a_1+a_5) \frac{\kk_2\kk_4}{\kk_3} } }{2}, \\
        \kk_5 &= \frac{2a_1 + 3a_5 - \sqrt{a_5^2 + 4(a_1+a_5) \frac{\kk_2\kk_4}{\kk_3} }}{2} , \\
        \kk_6 & = \frac{-2a_1-a_5 + \sqrt{a_5^2 + 4(a_1+a_5) \frac{\kk_2\kk_4}{\kk_3} } }{2}, 
    }
    so that $\kk_1 + \kk_5 = a_1 + a_5$ and $\kk_5 + \kk_6 = a_5$,  and \eqref{eq:ex-MMT} hold. The former means that the systems in \Cref{fig:square-diag-Ex-network,fig:square-diag-Ex-DE} share the same ODEs, while \eqref{eq:ex-MMT} implies the system in \Cref{fig:square-diag-Ex-DE} is complex-balanced. Applying \Cref{cor:specialcases} to the system in \Cref{fig:square-diag-Ex-DE}, we can conclude global stability on a small neighbourhood of 
    \eq{ 
        a_1 \leq  \frac{\kk_2\kk_4}{\kk_3} \leq a_1 + 2a_5,
    }
    the region bounded by the two surfaces in \Cref{fig:square-diag-Ex-param}. 
    Of course, since \eqref{eq:ex-original} and \eqref{eq:ex-DE} are identical, global stability holds for the system in \Cref{fig:square-diag-Ex-network} as well. In short, by allowing for dynamical equivalence, it is possible to extend the known region for global stability, in this case from near the lower surface in \Cref{fig:square-diag-Ex-param} to a neighbourhood around the region bounded by the two surfaces in the figure. 
\end{ex}

\printbibliography

\end{document}